\documentclass[smallextended]{svjour2}

\usepackage{amsmath, amstext, amssymb, amsfonts, graphicx, bm, subfigure}

\DeclareMathOperator{\Var}{\mathrm{Var}}
\newcommand{\PR}{\mathbb P}
\newcommand{\QR}{\mathbb Q}
\newcommand{\ER}{\mathbb E}

\DeclareMathOperator{\R}{\mathbb{R}}

\journalname{Journal of Statistical Physics}

\begin{document}

\title{Phase transitions in edge-weighted exponential random graphs: Near-degeneracy and universality}

\author{Ryan DeMuse \and Danielle Larcomb \and Mei Yin\thanks{Mei Yin's research was partially
supported by NSF grant DMS-1308333.}}

\institute{Department of Mathematics, University of
Denver, Denver, CO 80208, USA \\
              \email{mei.yin@du.edu}\\
}

\date{Received: date / Accepted: date}

\maketitle

\begin{abstract}
Conventionally used exponential random graphs cannot directly model weighted networks as the underlying probability space consists of simple graphs only. Since many substantively important networks are weighted, this limitation is especially problematic. We extend the existing exponential framework by proposing a generic common distribution for the edge weights. Minimal assumptions are placed on the distribution, that is, it is non-degenerate and supported on the unit interval. By doing so, we recognize the essential properties associated with near-degeneracy and universality in edge-weighted exponential random graphs.


\end{abstract}

\section{Introduction}
\label{intro}
Large networks have become increasingly popular over the last decades, and their modeling and investigation have led to interesting and new ways to apply statistical and analytical methods. Much of the random graph literature has evolved from the famous Erd\H{o}s-R\'{e}nyi graph, where edges are joined between vertices independently with the same probability. While the simple formation has attracted significant mathematical interest, this construction lacks the ability to model real world networks, which exhibit many noticeable attributes such as clustering and transitivity. The introduction of exponential random graphs has aided in this pursuit as they are able to capture a wide variety of common network tendencies by representing a complex global structure through a set of tractable local features \cite{FS} \cite{Newman} \cite{WF}. See Besag \cite{Besag}, Snijders et al. \cite{SPRH}, Rinaldo et al. \cite{RFZ}, and Fienberg \cite{Fienberg1} \cite{Fienberg2} for history and a review of developments.

These rather general models are exponential families of probability distributions over graphs, in which dependence between the random edges is defined through certain finite subgraphs. Inquiries into exponential random graphs have been made on the variational principle of the limiting normalization constant, concentration of the limiting probability distribution, phase transitions, and asymptotic structures. See for example Chatterjee and Varadhan \cite{CV}, Chatterjee and Diaconis \cite{CD1}, Radin and Yin \cite{RY}, Lubetzky and Zhao \cite{LZ1} \cite{LZ2}, Radin and Sadun \cite{RS1} \cite{RS2}, Radin et al. \cite{RRS}, Kenyon et al. \cite{KRRS}, Yin \cite{Yin2013}, Kenyon and Yin \cite{KY}, Aristoff and Zhu \cite{AZ2}, and Chatterjee and Dembo \cite{CD2}. Many of these papers utilize the elegant theory of graph limits as developed by Lov\'{a}sz and coauthors (V.T. S\'{o}s, B. Szegedy, C. Borgs, J. Chayes, K. Vesztergombi, \ldots) \cite{BCLSV1} \cite{BCLSV2} \cite{BCLSV3} \cite{Lov} \cite{LS}. Building on earlier work of Aldous \cite{Aldous1} and Hoover \cite{Hoover}, the graph limit theory creates a new set of tools for representing and studying the asymptotic behavior of graphs by connecting sequences of graphs $G_n$, which are discrete objects that lie in different probability spaces, to a unified graphon space $\mathcal{W}$, which is an abstract functional space equipped with a cut metric. Though the theory itself is tailored to dense graphs whose number of edges scales like the square of number of vertices, parallel theories for sparse graphs are likewise emerging. See Benjamini and Schramm \cite{BS}, Aldous and Steele \cite{AS}, Aldous and Lyons \cite{AL}, and Lyons \cite{Lyons} where the notion of local weak convergence is discussed and the recent works of Borgs et al. \cite{BCCZ1} \cite{BCCZ2} that are making progress towards enriching the existing $L^\infty$ theory of dense graph limits by developing a limiting object for sparse graph sequences based on $L^p$ graphons.

Despite their flexibility, conventionally used exponential random graphs suffer from some deficiencies that may hamper their utility to researchers. The major shortcomings are degeneracy problems, a sensitivity to missing data, and an inability to model weighted networks \cite{CD}. Since the underlying probability space of the standard exponential random graph model consists of simple graphs only yet many substantively important networks arising from a host of applications including socio-econometric data and neuroscience are weighted, this limitation is especially problematic. Consider a social network graph with vertices being people and edges indicating a relationship. We contemplate that family members have stronger relationships with one another than do workplace colleagues. This can be reflected by placing a weight on the edges to demonstrate some prior belief in the strength of connection, with coworkers having low weighted edges and family members having high weighted edges in-between. Properly adjusting the edge weights thus allows the modeling of a broad range of networks, be it consisting of more familial ties or more acquaintances.

An alternative interpretation for simple graphs is such that the edge weights are iid and satisfy a Bernoulli distribution. Following this perspective, Yin \cite{Yin2016} extended the exponential framework by putting a generic common distribution on the iid edge weights. After deriving a variational principle for the limiting normalization constant and an associated concentration of measure, an explicit characterization of the asymptotic phase transition was obtained for exponential models with uniformly distributed edge weights. This work expands upon the setting in \cite{Yin2016} and places minimal assumptions on the edge-weight distribution, that is, it is non-degenerate and supported on the unit interval. By doing so, we strive to discover universal asymptotic behavior, i.e. behavior that does not depend on the particular edge-weight distribution, for the model in the near-degenerate regions of the parameter space corresponding to where the graph is sparse (almost entirely unconnected) or nearly complete (almost fully connected) \cite{CD} \cite{Handcock} \cite{Yin}.

The rest of this paper is organized as follows. In Section \ref{background} we provide basics of graph limit theory and introduce key features of edge-weighted exponential random graphs. In Section \ref{Legendre} we summarize important properties of Legendre duality between the cumulant generating function and the Cram\'{e}r rate function for the edge-weight distribution. In Section \ref{maximization} we show the existence of a first order phase transition curve ending in a second order critical point in general edge-weighted exponential random graph models through a detailed analysis of a maximization problem for the normalization constant. Lastly, in Section \ref{asymptotics} we explore the universal and non-universal asymptotics concerning the phase transition.

\section{Background}
\label{background}
Consider the set $\mathcal{G}_n$ of all simple edge-weighted
complete labeled graphs $G_n$ on $n$ vertices (``simple'' means undirected, with no loops or multiple edges), where the edge weights
$x_{ij}$ between vertex $i$ and vertex $j$ are iid real random
variables satisfying a non-degenerate common distribution $\mu$ that is supported on $[0, 1]$. Any such graph
$G_n$, irrespective of the number of vertices, may be represented
as an element $h^{G_n}$ of a single abstract space $\mathcal{W}$
that consists of all symmetric measurable functions $h(x,y)$ from the unit square $[0,1]^2$ into the unit interval $[0, 1]$ (referred to as ``graph limits'' or ``graphons''), by setting $h^{G_n}(x,y)$ as the edge weight
between vertices $\lceil nx \rceil$ and $\lceil ny \rceil$ of
$G_n$. The common distribution $\mu$ for the
edge weights yields probability measure $\PR_n$ and the associated
expectation $\ER_n$ on $\mathcal{G}_n$, and further induces
probability measure $\QR_n$ on the space $\mathcal{W}$ under the
graphon representation.

For a finite simple graph $H$ with vertex set
$V(H)=[k]=\{1,...,k\}$ and edge set $E(H)$ and a simple graph
$G_n$ on $n$ vertices, there is a notion of density of graph
homomorphisms, denoted by $t(H, G_n)$, which indicates the
probability that a random vertex map $V(H) \to V(G_n)$ is
edge-preserving,
\begin{equation}
\label{t} t(H, G_n)=\frac{|\text{hom}(H,
G_n)|}{|V(G_n)|^{|V(H)|}}.
\end{equation}
For a graphon $h\in \mathcal{W}$, define the graphon homomorphism
density
\begin{equation}
\label{tt} t(H, h)=\int_{[0,1]^k}\prod_{\{i,j\}\in E(H)}h(x_i,
x_j)dx_1\cdots dx_k.
\end{equation}
Then $t(H, G_n)=t(H, h^{G_n})$ by construction, and we take
(\ref{tt}) with $h=h^{G_n}$ as the definition of graph homomorphism density $t(H,
G_n)$ for an edge-weighted complete graph $G_n$. This graphon
interpretation enables us to capture the notion of convergence in
terms of subgraph densities by an explicit ``cut distance'' on
$\mathcal{W}$:
\begin{equation}
d_{\square}(f, h)=\sup_{S, T \subseteq [0,1]}\left|\int_{S\times
T}\left(f(x, y)-h(x, y)\right)dx\,dy\right|
\end{equation}
for $f, h \in \mathcal{W}$. Except for a technical complication explained below, a sequence of edge-weighted
graphs converges under the cut metric if and only if its homomorphism densities
converge for all finite simple graphs, and the limiting
homomorphism densities then describe the resulting graphon.

The technical complication is that the topology induced by the cut
metric is well defined only up to measure preserving
transformations of $[0,1]$ (and up to sets of Lebesgue measure
zero), which may be thought of as a vertex relabeling in the context of
finite graphs. To tackle this issue, an equivalence relation
$\sim$ is introduced in $\mathcal{W}$. We say that $f\sim h$ if
$f(x, y)=h_{\sigma}(x, y):=h(\sigma x, \sigma y)$ for some measure
preserving bijection $\sigma$ of $[0,1]$. Let $\tilde{h}$
(referred to as a ``reduced graphon'') denote the equivalence
class of $h$ in $(\mathcal{W}, d_{\square})$. Since $d_{\square}$
is invariant under $\sigma$, one can then define on the resulting
quotient space $\widetilde{\mathcal{W}}$ the natural distance
$\delta_{\square}$ by $\delta_{\square}(\tilde{f},
\tilde{h})=\inf_{\sigma_1, \sigma_2}d_{\square}(f_{\sigma_1},
h_{\sigma_2})$, where the infimum ranges over all measure
preserving bijections $\sigma_1$ and $\sigma_2$, making
$(\widetilde{\mathcal{W}}, \delta_{\square})$ into a metric space.
With some abuse of notation we also refer to $\delta_{\square}$ as
the ``cut distance''. After identifying graphs that are the same
after vertex relabeling, the probability measure $\PR_n$ yields
probability measure $\tilde{\PR}_n$ and the associated expectation
$\tilde{\ER}_n$ (which coincides with $\ER_n$). Correspondingly,
the probability measure $\QR_n$ induces probability measure
$\tilde{\QR}_n$ on the space $\widetilde{\mathcal{W}}$ under the
measure preserving transformations. The space $(\widetilde{\mathcal{W}},
\delta_{\square})$ is a compact space and homomorphism densities $t(H, \cdot)$ are
continuous functions on it.

By a $2$-parameter family of edge-weighted exponential random graphs we mean a
family of probability measures $\PR_n^{\beta}$ on $\mathcal{G}_n$
defined by, for $G_n\in\mathcal{G}_n$,
\begin{equation}
\label{pmf} \PR_n^{\beta}(G_n)=\exp\left(n^2\left(\beta_1
t(H_1,G_n)+\beta_2 t(H_2,G_n)-\psi_n^{\beta}\right)\right)\PR_n(G_n),
\end{equation}
where $\beta=(\beta_1, \beta_2)$ are $2$ real parameters, $H_1$ is
a single edge, $H_2$ is a finite simple graph with $p\geq 2$
edges, $t(H_i, G_n)$ is the density of graph
homomorphisms, $\PR_n$ is the probability measure induced by the
common distribution $\mu$ for the edge weights, and
$\psi_n^{\beta}$ is the normalization constant (free energy density),
\begin{equation}
\label{psi} \psi_n^{\beta}=\frac{1}{n^2}\log \ER_n
\left(\exp\left(n^2 \left(\beta_1 t(H_1,G_n)+\beta_2
t(H_2,G_n)\right) \right)\right).
\end{equation}
Since homomorphism densities $t(H_i, G_n)$ are
preserved under vertex relabeling, the probability measure
$\tilde{\PR}_n^\beta$ and the associated expectation
$\tilde{\ER}_n^\beta$ (which coincides with $\ER_n^\beta$) may
likewise be defined.

Being exponential families with bounded support, one might expect exponential random graph models to enjoy a rather basic asymptotic form, though
in fact, virtually all these models are highly nonstandard as $n$ increases. The $2$-parameter edge-weighted exponential random graph models are
simpler than their $k$-parameter extensions but nevertheless exhibit a wealth of non-trivial characteristics and capture a variety of interesting features displayed by large networks. Furthermore, the relative simplicity provides insight into the expressive power of the exponential construction. In statistical physics, we refer to $\beta_1$ as the particle parameter and $\beta_2$ as the energy parameter. Accordingly, the exponential model (\ref{pmf}) is said to be ``attractive'' if $\beta_2$ is positive and ``repulsive'' if $\beta_2$ is negative. In this paper we will concentrate on ``attractive'' $2$-parameter models. The interest in these models is well justified. Consider the friendship graph for example, where the edge weights between different vertex pairs measure the strength of mutual friendship. Take $H_1$ an edge and $H_2$ a triangle. Since a friend of a friend is likely also a friend, the influence of a triangle that assesses the bond of a $3$-way friendship should be emphasized, and this corresponds to taking $\beta_2 \geq 0$. The edge-triangle model thus captures transitivity when $n$ is finite, but this transitivity is gradually lost when $n$ tends to infinity in the sense that the model produces a graph that looks similar to an Erd\H{o}s-R\'{e}nyi graph with respect to the cut metric (see detailed discussions in Sections \ref{maximization} and \ref{asymptotics}).


\section{Legendre transform and duality}
\label{Legendre}
In this section we present properties of the cumulant generating function $K(\theta)$ and the Cram\'{e}r rate function $I(u)$ for the edge-weight distribution $\mu$ relevant to our investigation. We will see that $K(\theta)$ is convex on $\R$, which allows the application of the Legendre transform. Let $I: A \to \R$ be the Legendre transform of $K$ given by
\begin{equation}
\label{cramer}
I(u) = \sup_{\theta \in \R} \left\{ \theta u - K(\theta) \right\},
\end{equation}
where $A$, the domain of $I$, consists of all $u$ so that $I(u)<\infty$. Note that in large deviation theory, $I$ is commonly referred to as the Cram\'{e}r conjugate rate function for the distribution $\mu$. It follows from theorems proved in Chapter 2: Analytic Properties of \cite{Brown} that the Legendre transform connecting $K$ and $I$ is an involution, $I$ is smooth and strictly convex everywhere it is defined, and there is a 1-1 relationship between $K$ and $I$. Lemma \ref{KK} and Proposition \ref{prop} then discuss properties of $K(\theta)$ and $I(u)$ under the additional assumption that $\mu$ is symmetric. These properties will be useful in Section \ref{asymptotics} when we explore universality in edge-weighted exponential random graphs.






\begin{lemma}
\label{KK}
Consider a non-degenerate probability measure $\mu$ supported on $[0, 1]$. Let $K(\theta)$ be the associated cumulant generating function. If $\mu$ is symmetric about the line $u=1/2$, then $K'''(0)K'(0)+(p-2)\left(K''(0)\right)^2 \geq 0$, and equality is obtained only when $p=2$.
\end{lemma}

\begin{proof}
Let $X$ be a random variable distributed according to $\mu$. By symmetry, $\mathbb{E}(X)=1/2$ and $\mathbb{E}(X^3)=3\mathbb{E}(X^2)/2-1/4$. This implies that $K'(0)=\mathbb{E}(X)=1/2$ and $K''(0)=\mathbb{E}(X^2)-\left(\mathbb{E}(X)\right)^2=\mathbb{E}(X^2)-1/4$. Also,
\begin{equation}
K'''(0) = \mathbb{E}(X^3) - 3\mathbb{E}(X^2)\mathbb{E}(X) + 2\left(\mathbb{E}(X)\right)^3=0.
\end{equation}
The claim thus follows.
\end{proof}

\begin{lemma}
\label{domain}
Consider a non-degenerate probability measure $\mu$ supported on $[0, 1]$. Let $I(u)$ be the associated Cram\'{e}r rate function (\ref{cramer}). Then the domain of $I$ is a subset of $[0, 1]$.
\end{lemma}

\begin{proof}
Since $\mu$ is supported on $\left[0,1\right]$, we have $0 \le K(\theta) \le \theta$ if $\theta \ge 0$, and $\theta \le K(\theta) \le 0$ if $\theta \le 0$.
This gives
\begin{align}
I(u) &= \sup{\left\{ \sup_{\theta \ge 0}\left\{ {\theta}u - K(\theta) \right\} , \sup_{\theta \le 0}\left\{ {\theta}u - K(\theta) \right\}  \right\}} \\
	&\ge \sup{\left\{ \sup_{\theta \ge 0}\left\{ {\theta}\left(u - 1\right) \right\} , \sup_{\theta \le 0}\left\{ {\theta}u \right\} \right\}}. \nonumber
\end{align}
If $u > 1$ then $\sup_{\theta \ge 0} \left\{{\theta}\left(u-1\right) \right\} = \infty$ and thus $I(u)$ is not finite. Similarly, if $u < 0$ then $\sup_{\theta \le 0}\left\{{\theta}u \right\} = \infty$ and thus $I(u)$ is not finite. The conclusion readily follows.
\end{proof}




\begin{table}
\begin{equation*}
\begin{array}{cc}
\text{Limiting Properties of } K(\theta) & \theta \text{ limit } \\
\hline \hline \\
K(\theta) \to -\infty \text{ or } l < 0 & \theta \to -\infty \\
K'(\theta) \to 0 & \theta \to -\infty \\
K''(\theta) \to 0 & \theta \to -\infty \\
K(\theta) \to \infty & \theta \to \infty \\
K'(\theta) \to 1 & \theta \to \infty \\
K''(\theta) \to 0 & \theta \to \infty \\ \\
\end{array}
\end{equation*}
\caption{Limiting properties of $K(\theta)$ as $\theta \rightarrow \pm \infty$.} \label{table}
\end{table}

Analyzing properties of $K(\theta)$ and $I(u)$ in detail will give a stronger conclusion than Lemma \ref{domain}. We recognize that the cumulant generating function $K(\theta)$ satisfies $K(0) = 0$, $K'(0) = \mathbb{E}(X)$, and $K''(0)=\Var(X)$, where $X$ is a random variable distributed according to $\mu$. See Table \ref{table} for important limiting properties of $K(\theta)$ as $\theta \rightarrow \pm \infty$. By Legendre duality, every $u \in (0,1)$ uniquely corresponds to a $\theta \in \left(-\infty,\infty\right)$, with $K'(\theta)=u$ and $I'(u)=\theta$. This implies that $I(\mathbb{E}(X))=I'(\mathbb{E}(X))=0$, and $I(u)$ is decreasing on $(0,\mathbb{E}(X))$ and increasing on $(\mathbb{E}(X),1)$. We also note that $I(0)$ and $I(1)$, depending on the probability distribution $\mu$, may be either finite or grow unbounded. In the former case, the domain of $I$ is $[0, 1]$ (as for Bernoulli$(.5)$). In the latter case, the domain of $I$ is $(0, 1)$ (as for Uniform$(0, 1)$).


\begin{proposition}
\label{prop}
Consider a non-degenerate probability measure $\mu$ supported on $[0, 1]$. Let $I(u)$ be the associated Cram\'{e}r rate function (\ref{cramer}). If $\mu$ is symmetric about the line $u=1/2$, then $I(u)$ is also symmetric about the line $u=1/2$.
\end{proposition}

\begin{proof}
Let $\theta \in \R$. Under the symmetry assumption, we will show, by a simple change of variable $x = 1 - y$, that $K(-\theta) = -\theta + K(\theta)$.
\begin{align}
\label{K}
K(-\theta) &= \log\int e^{-\theta x} \mu(dx) = \log\int e^{-\theta (1-y)} \mu(dy) \\
	&= \log\int e^{-\theta} e^{\theta y} \mu(dy) = -\theta + K(\theta). \nonumber
\end{align}
Let $u \in (0,1)$. Following Legendre duality, $u = K'(\theta)$ for a unique $\theta$. By (\ref{K}), this implies that $1-u=1-K'(\theta)=K'(-\theta)$, i.e., $1-u$ and $-\theta$ are unique duals of each other. We compute
\begin{align}
I(u) &= \theta K'(\theta) - K(\theta) \\
	&= \theta \left(1-K'(-\theta)\right) - \left( K(-\theta) + \theta \right) \nonumber \\
	&= (-\theta) K'(-\theta) - K(-\theta)=I(1-u). \nonumber
\end{align}
This verifies our claim.
\end{proof}


\section{Maximization analysis}
\label{maximization}
In this section we demonstrate the existence of first order phase transitions in general edge-weighted exponential random graphs. Our main results are Theorem \ref{phase} and the consequent Corollary \ref{corollary}. In the standard statistical physics literature, phase transition is often associated with loss of analyticity in the normalization constant, which gives rise to discontinuities in the observed graph
statistics. In the vicinity of a phase transition, even a tiny change in some local feature can result in a dramatic change of the entire system.

\begin{definition}
A phase is a connected region of the parameter space $\{\beta\}$,
maximal for the condition that the limiting normalization constant
$\psi_\infty^\beta:=\lim_{n\rightarrow \infty} \psi_n^\beta$ is analytic. There is a $j$th-order transition
at a boundary point of a phase if at least one $j$th-order partial
derivative of $\psi_\infty^\beta$ is discontinuous there, while
all lower order derivatives are continuous.
\end{definition}

Following this philosophy, we will make use of two theorems from \cite{Yin2016}, which connect the occurrence of an asymptotic phase transition in
our model with the solution of a certain maximization problem for the limiting normalization constant.

\begin{theorem}[Theorem 3.4 in \cite{Yin2016}]
\label{CD} Consider a general $2$-parameter exponential random
graph model (\ref{pmf}). Suppose $\beta_2$ is
non-negative. Then the limiting normalization constant
$\psi_\infty^\beta$ exists, and is given by
\begin{equation}
\label{lmax} \psi_{\infty}^{\beta}=\sup_{u}\left(\beta_1
u+\beta_2 u^{p}-\frac{1}{2}I(u)\right)\\,
\end{equation}
where $H_2$ is a simple graph with $p\geq 2$ edges, $I$ is the
Cram\'{e}r rate function (\ref{cramer}), and the supremum is
taken over all $u$ in the domain of $I$, i.e., where $I<\infty$.
\end{theorem}

\begin{theorem}[Theorem 3.5 in \cite{Yin2016}]
\label{gen} Let $G_n$ be an exponential random graph drawn from
(\ref{pmf}). Suppose $\beta_2$ is non-negative. Then
$G_n$ behaves like an Erd\H{o}s-R\'{e}nyi graph $G(n, u)$ in the
large $n$ limit:
\begin{equation}
\label{behave}
\lim_{n\rightarrow \infty}\delta_\square(\tilde{h}^{G_n},
\tilde{u})=0 \text{ almost surely},
\end{equation}
where $u$ is picked randomly from the set $U$ of
maximizers of (\ref{lmax}).
\end{theorem}

To be more precise, Theorems \ref{CD} and \ref{gen} indicate that a typical graph drawn from the exponential random graph model is {\it weakly pseudorandom} \cite{BBS}. Weakly pseudorandomness means that with exponentially high probability, a sampled graph satisfies a number of equivalent properties such as large spectral gap and correct number of all subgraph counts that make it very similar to an Erd\H{o}s-R\'enyi graph. Some authors have delved deeper into this ``asymptotically equivalent'' phenomenon. Mukherjee \cite{M} considered the two star model in \cite{M} and found that though the model looks like an Erd\H{o}s-R\'enyi mixture in cut distance, the same convergence does not go through in total variation. This says that despite being very close to Erd\H{o}s-R\'enyi, a graph sampled from the exponential distribution is not exactly Erd\H{o}s-R\'enyi. In the case of edge-triangle model, Radin and Sadun \cite{RS2} argued that the two-parameter to one-parameter reduction and the loss of information is essentially due to the inequivalence of grand canonical and microcanonical ensembles of the exponential model in the asymptotic regime. From a practical perspective, however, the Erd\H{o}s-R\'enyi approximation for the exponential random graph is already good enough, and we may simply picture an exponential random graph as an Erd\H{o}s-R\'enyi graph in the large graph ``attractive'' limit \cite{CD1}.

A significant part of computing phase boundaries for the $2$-parameter exponential model is then
a detailed analysis of a calculus problem coupled with probability
estimates. However, as straightforward as it sounds, since
the exact form of the Cram\'{e}r rate function $I$ is not readily obtainable for a generic edge-weight distribution $\mu$, getting a clear picture of the asymptotic phase structure is not that easy and various tricks, especially the duality principle for the Legendre transform, need to be employed \cite{ZRM}. We note that our mechanism for $2$-parameter models may be further generalized to a $k$-parameter setting, and the crucial idea is to minimize the effect of the ordered parameters on
the limiting normalization constant one by one. See \cite{Yin2013} for an illustration of this procedure in the standard exponential random graph model (where $\mu$ is Bernoulli$(.5)$).

\vskip.1truein

\noindent \textbf{Assumption}
Let $p$ be the number of edges in $H_2$. Denote by $K(\theta)$ the cumulant generating function associated with the probability measure $\mu$.
We place a technical assumption:
\begin{equation}
K'''(\theta)K'(\theta) = -(p-2)\left( K''(\theta) \right)^2
\end{equation}
admits only one zero on $\R$.

We remark that this requirement on $\mu$, which is satisfied by many common distributions including Bernoulli$(.5)$ and Uniform$(0, 1)$ etc., is just a technicality that\ guarantees the existence of a unique phase transition curve. Without this assumption, there may be more than one phase transition curve. Still, all phase transition curves display the same asymptotical behavior as described in (\ref{erd}), and all graph samples drawn from the ``attractive'' region of the parameter space are approximately Erd\H{o}s-R\'enyi (but with varying densities). The parameter space therefore consists of a single (Erd\H{o}s-R\'enyi) phase with first order phase transition(s) across one (or more) curves and second order phase transition(s) along the boundaries, and the transitions correspond to a change in density of the Erd\H{o}s-R\'enyi graph.

The meaning of a phase transition in the exponential model thus deserves some careful re-examination. As will be shown in Theorem \ref{phase}, there are curves approaching the phase transition curve from either side along which the corresponding weakly pseudorandom Erd\H{o}s-R\'{e}nyi distribution stays constant, and a jump in the Erd\H{o}s-R\'{e}nyi parameter $u$ occurs only when the phase transition curve is crossed. This implies that asymptotically the state of the network (represented by $u$) does not have a one-to-one correspondence with the associated exponential parameter $\beta$. (The same defect was observed by Chatterjee and Diaconis \cite{CD1} in the unweighted situation.) Some intricate differences between the exponential model and the related Erd\H{o}s-R\'{e}nyi model are presented in Sections \ref{maximization} and \ref{asymptotics}, particularly through the calculations after Theorem \ref{lim2}. The last equation (\ref{last}) offers a possible way of distinguishing among ``equivalent'' exponential parameters $\beta$, since same Erd\H{o}s-R\'{e}nyi parameter $u$ but different model parameter $\beta_2$ lead to different limiting normalization constant in the exponential model, which encodes important asymptotic information about the system.

Given an observed network that one wishes to model using an exponential random graph model, there may be many parameter values yielding the same weakly pseudorandom Erd\H{o}s-R\'{e}nyi distribution, and practitioners need to determine what is a best choice. Ideally, those parameters would generate a model whose measurements from simulated realizations reflect the observed network as accurately as possible in every aspect (not just the correct number of subgraph counts as determined by the Erd\H{o}s-R\'{e}nyi parameter). Restrictions on the run time of the data collection process may be further imposed. These practical considerations have led to continued interest and advances both theoretically and experimentally in improving goodness of fit and parameter learning \cite{Handcock} \cite{Hunter} and developing better model specifications \cite{SPRH}. For a general principle, good models should produce networks that are structurally similar to the observed network using few but effective parameters, while bad models produce networks that bear little resemblance to the observed network using many unnecessary parameters.

\begin{theorem}
\label{phase}
Suppose the common distribution $\mu$ for the edge weights is supported on $[0, 1]$ and non-degenerate. For any allowed $H_2$, the limiting normalization
constant $\psi_\infty^\beta$ of (\ref{pmf}) is analytic at all $(\beta_1,
\beta_2)$ in the upper half-plane $(\beta_2\geq 0)$ except on a
certain decreasing curve $\beta_2=r(\beta_1)$ which includes the
endpoint $(\beta_1^c, \beta_2^c)$. The derivatives
$\frac{\partial}{\partial \beta_1}\psi_\infty^\beta$ and
$\frac{\partial}{\partial \beta_2}\psi_\infty^\beta$ have (jump)
discontinuities across the curve, except at the end point where all the second derivatives $\frac{\partial^2}{\partial
\beta_1^2}\psi_\infty^\beta$, $\frac{\partial^2}{\partial \beta_1
\partial \beta_2}\psi_\infty^\beta$ and $\frac{\partial^2}{\partial
\beta_2^2}\psi_\infty^\beta$ diverge.
\end{theorem}

\begin{corollary}
\label{corollary}
For any allowed $H_2$, the parameter space $\{(\beta_1, \beta_2):
\beta_2\geq 0\}$ consists of a single phase with a first order
phase transition across the indicated curve $\beta_2=r(\beta_1)$
and a second order phase transition at the critical point
$(\beta_1^c, \beta_2^c)$, qualitatively like the gas/liquid transition in equilibrium materials.
\end{corollary}

\begin{figure}
\centering
\includegraphics[clip=true, height=3in]{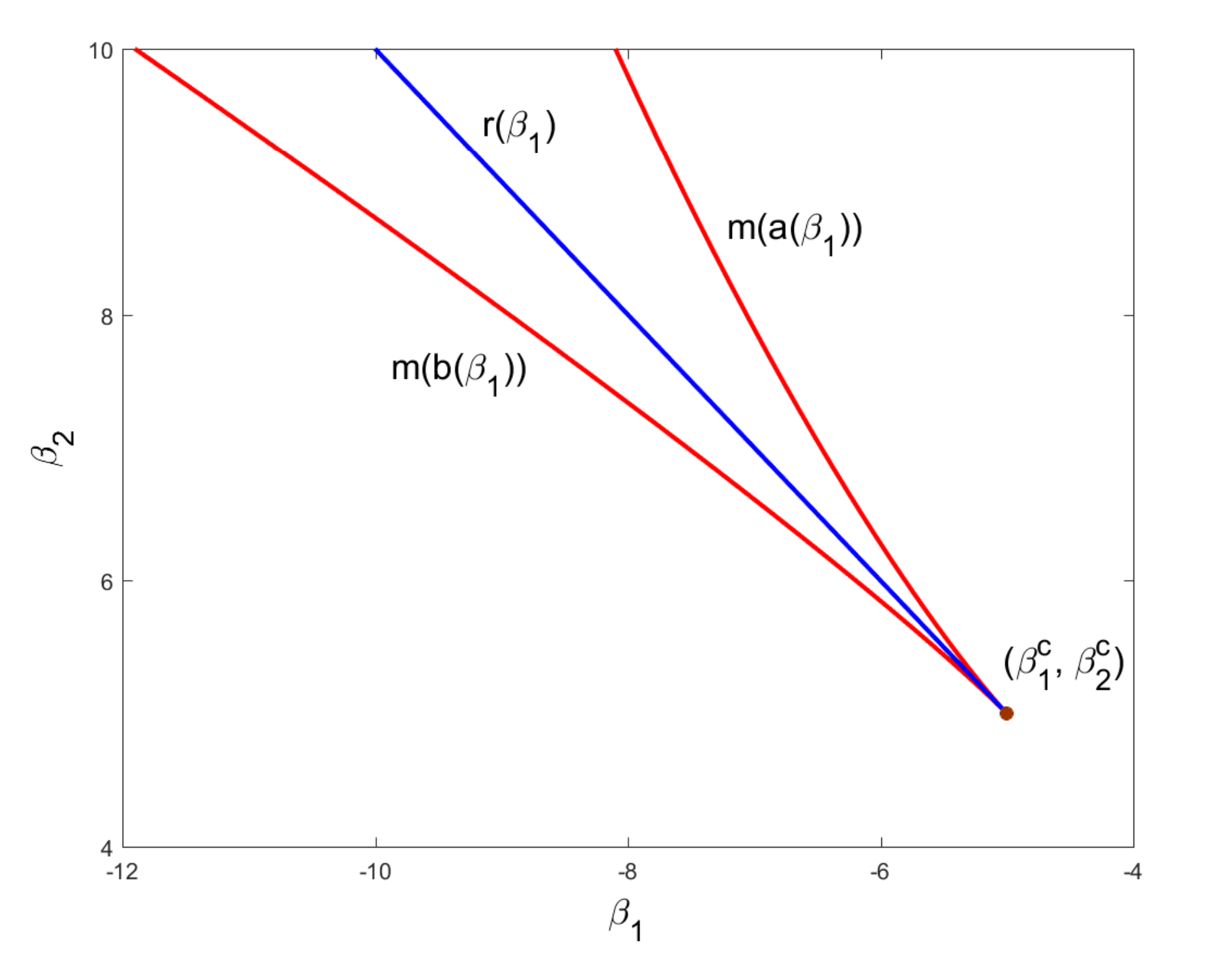}
\caption{The V-shaped region (with phase transition curve $r(\beta_1)$ inside) for the Beta$(2, 2)$ distribution in the $(\beta_1, \beta_2)$ plane. Graph drawn for $p=2$.} \label{Vshape}
\end{figure}

\vskip.1truein

\noindent \textit{Proof of Theorem \ref{phase}}
Let $p$ be the number of edges in $H_2$. Denote by $I(u)$ the Cram\'{e}r rate function associated with the probability measure $\mu$.
Define
\begin{equation}
\label{smax}
L(u; \beta_1, \beta_2) = \beta_1 u + \beta_2 u^{p} - \frac{1}{2} I(u)
\end{equation}
for $u \in [0,1]$. We consider the maximization problem for $L(u; \beta_1, \beta_2)$ on the interval $[0,1]$, where $-\infty < \beta_1 < \infty$ and $0 \le \beta_2 < \infty$ are parameters. We note that by Theorem \ref{CD}, the supremum should actually be taken over the domain of $I$, which might differ from $[0, 1]$ at the endpoints from the discussion following Lemma \ref{domain}. However, when the domain of $I$ does not include $0$ (or $1$), $L(0)$ (or $L(1)$) is negative infinity and so can not be the maximum. To locate the maximizers of $L(u)$, we examine the properties of $L'(u)$ and $L''(u)$,
\begin{equation}
L'(u) = \beta_1 + p\beta_2 u^{p-1} - \frac{1}{2} I'(u),
\end{equation}
\begin{equation*}
L''(u) = p(p-1)\beta_2 u^{p-2} - \frac{1}{2} I''(u).
\end{equation*}

Utilizing the duality principle for the Legendre transform between $I(u)$ and $K(\theta)$, we first analyze properties of $L''(u)$ on the interval $(0, 1)$. As a consequence of the Legendre transform,
\begin{equation}
\label{dual}
I(u) + K(\theta) = \theta u,
\end{equation}
where $\theta$ and $u$ are unique duals of each other. Taking derivatives, we find that
\begin{equation}
\label{duality}
u = K'(\theta) \hspace{.5cm} \text{ and } \hspace{.5cm} I''(u)K''(\theta) = 1.
\end{equation}
Consider the function
\begin{equation} \label{m} m(u) = \frac{I''(u)}{2p(p-1)u^{p-2}} \end{equation}
on $(0, 1)$. By (\ref{duality}), we may analyze the properties of $m(u)$ through the function
\begin{equation} n(\theta) = 2p(p-1)K''(\theta)\left(K'(\theta)\right)^{p-2}, \end{equation}
where $\theta \in \R$ and $m(u)n(\theta) = 1$. From the discussion following Lemma \ref{domain}, we recognize that
\begin{equation} \label{ntheta} \lim_{n \to -\infty} n(\theta) = 0, \end{equation}
\[ \lim_{n \to 0} n(\theta) = 2p(p-1)\Var(X)\left(\mathbb{E}(X)\right)^{p-2}, \]
\[ \lim_{n \to \infty} n(\theta) = 0, \]
where $X$ is a random variable distributed according to $\mu$. Since
\begin{equation}
n'(\theta)= 2p(p-1)\left(K'(\theta)\right)^{p-3} \left( K'''(\theta) K'(\theta) + (p-2)\left(K''(\theta)\right)^{2} \right)
\end{equation}
and $K'(\theta)>0$ always, under {\textbf{Assumption}} there exists a unique $\theta_0$ such that $n'(\theta_0)=0$. This unique global maximizer $\theta_0$ for $n(\theta)$ corresponds to a unique global minimizer for $m(u)$, which we denote by $u_0$. Using duality, $m(u)>0$ for all $u \in (0, 1)$ and grows unbounded on both ends. For $\beta_2 \le m(u_0)$, $L''(u) \le 0$ on $(0,1)$. For $\beta_2 > m(u_0)$, $L''(u) < 0$ for $0 < u < u_1$ and $u_2 < u < 1$ and $L''(u) > 0$ for $u_1 < u < u_2$, where the transition points $u_1$ and $u_2$ satisfy $L''(u_1)=L''(u_2)=0$. Sign properties of $L''(u)$ translate to monotonicity properties of $L'(u)$ over $(0, 1)$. For $\beta_2 \le m(u_0)$, $L'(u)$ is decreasing over $(0, 1)$. For $\beta_2 > m(u_0)$, $L'(u)$ is decreasing from $0$ to $u_1$, increasing from $u_1$ to $u_2$, and decreasing from $u_2$ to $1$. See Figure \ref{figure} for an illustrative plot of $n(\theta)$ and $m(u)$.

\begin{figure}
\begin{center}$
\begin{array}{cc}
\includegraphics[scale=.8]{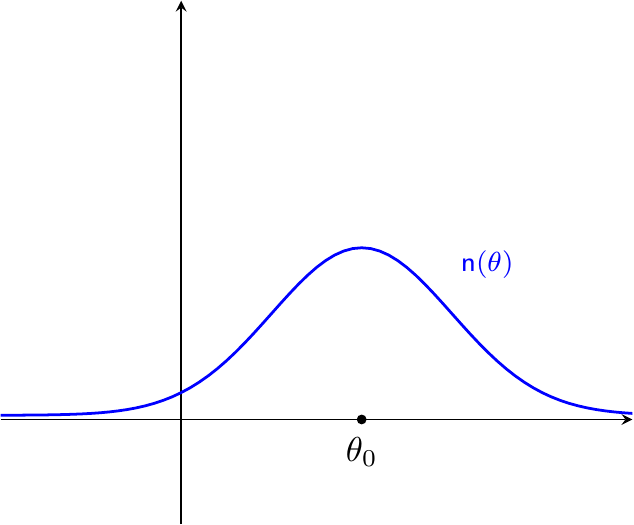} & \quad
\includegraphics[scale=.8]{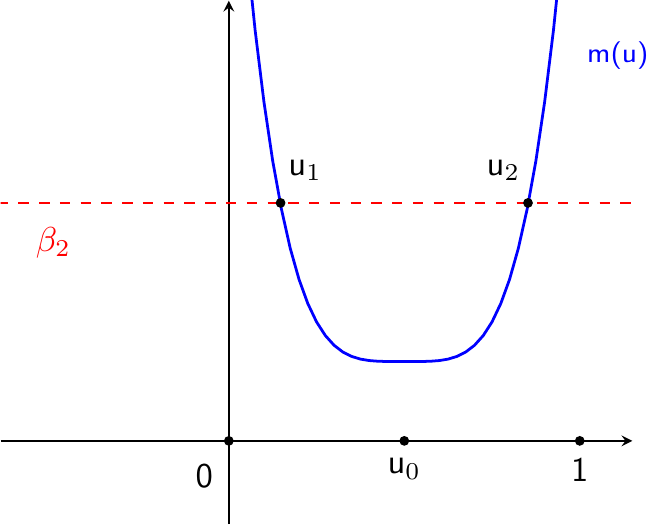}
\end{array}$
\end{center}
\caption{An illustrative plot of $n(\theta)$ and $m(u)$.}
\label{figure}
\end{figure}


The analytic properties of $L''(u)$ and $L'(u)$ entail analytic properties of $L(u)$ on the interval $[0, 1]$. Utilizing the duality of the Legendre transform (\ref{dual}) (\ref{duality}), $I(u)$ is a smooth convex function, $I'(0) = -\infty$ and $I'(1) = \infty$. Therefore $L'(0) = \infty$ and $L'(1) = -\infty$, so $L(u)$ cannot be maximized at $u = 0$ or $u = 1$. For $\beta_2 \le m(u_0)$, $L(u)$ is decreasing from $\infty$ at 0 to $-\infty$ at 1 passing the $u$-axis only once. This intercept, which we denote by $u^{\ast}$, is the unique global maximizer for $L(u)$. Now consider $\beta_2 > m(u_0)$. If $L'(u_1) \ge 0$, then $L'(u)$ has a unique zero greater than $u_2$ and so $L(u)$ has a unique global maximizer at $u^{\ast}>u_2$. If $L'(u_2) \le 0$, then $L'(u)$ has a unique zero less than $u_1$ and so $L(u)$ has a unique global maximizer at $u^{\ast}<u_1$.
Lastly, suppose that $L'(u_1) < 0 < L'(u_2)$. Then $L(u)$ has two local maximizers. Denote them by $u^{\ast}_1$ and $u^{\ast}_2$, with $0 < u^{\ast}_1 < u_1 < u_0 < u_2 < u^{\ast}_2 < 1$. See Figure \ref{figure2} for an illustrative plot of $L(u)$ in this case.

\begin{figure}
\begin{center}
\includegraphics[scale=1]{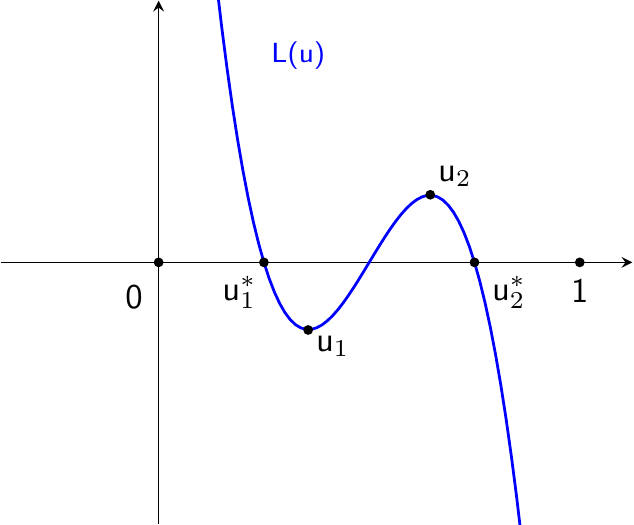}
\end{center}
\caption{An illustrative plot of $L(u)$ for $\beta_2>m(u_0)$.}
\label{figure2}
\end{figure}

Define
\begin{equation}
f(u) = \frac{uI''(u)}{2(p-1)} - \frac{1}{2}I'(u).
\end{equation}
Using $m(u_1)=m(u_2)=\beta_2$ (\ref{m}), $L'(u_1)=\beta_1+f(u_1)$ and $L'(u_2)=\beta_1+f(u_2)$. We compute
\begin{equation}
\label{fprime}
f'(u)=\frac{uI'''(u)-I''(u)(p-2)}{2(p-1)}=pu^{p-1}m'(u).
\end{equation}
As a consequence of the relation between $f'$ and $m'$, following the previous analysis for $m$, $f$ is decreasing on $(0, u_0)$ and increasing on $(u_0, 1)$. We check that similarly as $m$, $f$ grows unbounded on both ends. Taking $u\rightarrow 0$ corresponds to taking $\theta \rightarrow -\infty$ in the dual space (\ref{dual})(\ref{duality}), and the divergence is clear from the discussion following Lemma \ref{domain}. To see that $f(u)$ diverges as $u\rightarrow 1$, we utilize (\ref{fprime}). By the fundamental theorem of calculus,
\begin{equation}
f(u)-f(u_0)=\int_{u_0}^u f'(t)dt \geq pu_0^{p-1} \int_{u_0}^u m'(t)dt=pu_0^{p-1}(m(u)-m(u_0)),
\end{equation}
and grows to infinity as $u$ approaches $1$. Let $X$ be a random variable distributed according to $\mu$, we note some nice formulas for $f$ and $m$ for future reference:
\begin{equation}
\label{fm}
f(\mathbb{E}(X)) = \frac{\mathbb{E}(X)}{2(p-1)\Var(X)},
\end{equation}
\begin{equation*}
m(\mathbb{E}(X)) = \frac{1}{2p(p-1)\left(\mathbb{E}(X)\right)^{p-2} \Var(X)}.
\end{equation*}
In order for $L'(u_1) < 0$, we must have $\beta_1 < -f(u_1)$. Since $f$ attains an absolute minimum at $u_0$, $f(u_1) > f(u_0)$, and then $\beta_1 < -f(u_0)$. The only possible region in the $(\beta_1, \beta_2)$ plane where $L'(u_1) < 0 < L'(u_2)$ is thus bounded by $\beta_1 < -f(u_0)$ and $\beta_2 > m(u_0)$. Denote these two critical values for $\beta_1$ and $\beta_2$ by $\beta_1^c:= -f(u_0)$ and $\beta_2^c:= m(u_0)$.

Recall that $u_1<u_0<u_2$. By monotonicity of $f(u)$ on the intervals
$(0, u_0)$ and $(u_0, 1)$, there exist continuous functions
$a(\beta_1)$ and $b(\beta_1)$ of $\beta_1$, such that $L'(u_1)<0$
for $u_1>a(\beta_1)$ and $L'(u_2)>0$ for $u_2>b(\beta_1)$. As
$\beta_1\rightarrow -\infty$, $a(\beta_1)\rightarrow 0$ and
$b(\beta_1)\rightarrow 1$. $a(\beta_1)$ is an increasing function
of $\beta_1$, whereas $b(\beta_1)$ is a decreasing function, and
they satisfy $f(a(\beta_1))=f(b(\beta_1))=-\beta_1$. The
restrictions on $u_1$ and $u_2$ yield restrictions on $\beta_2$,
and we have $L'(u_1)<0$ for $\beta_2<m(a(\beta_1))$ and
$L'(u_2)>0$ for $\beta_2>m(b(\beta_1))$. As $\beta_1\rightarrow
-\infty$, $m(a(\beta_1))\rightarrow \infty$ and
$m(b(\beta_1))\rightarrow \infty$. $m(a(\beta_1))$ and
$m(b(\beta_1))$ are both decreasing functions of $\beta_1$, and
they satisfy $L'(u_1)=0$ when $\beta_2=m(a(\beta_1))$ and
$L'(u_2)=0$ when $\beta_2=m(b(\beta_1))$. As $L'(u_2)>L'(u_1)$ for
every $(\beta_1, \beta_2)$, the curve $m(b(\beta_1))$ lies
below the curve $m(a(\beta_1))$, and together they generate the
bounding curves of the $V$-shaped region in the $(\beta_1,
\beta_2)$ plane with corner point $(\beta_1^c, \beta_2^c)$ where
two local maximizers exist for $L(u)$. By (\ref{fprime}), for sufficiently negative values of $\beta_1$, $f(a(\beta_1))<m(a(\beta_1))$ and $f(b(\beta_1))>m(b(\beta_1))$, so the straight line $\beta_1=-\beta_2$ lies within this region.

Fix an arbitrary $\beta_1<\beta_1^c$. Then $L'(u)$ shifts upward as $\beta_2$ increases and downward as
$\beta_2$ decreases. As a result, as $\beta_2$ gets large, the
positive area bounded by the curve $L'(u)$ increases, whereas the
negative area decreases. By the fundamental theorem of calculus,
the difference between the positive and negative areas is the
difference between $L(u_2^*)$ and $L(u_1^*)$, which goes from
negative ($L'(u_2)=0$, $u_1^*$ is the global maximizer) to
positive ($L'(u_1)=0$, $u_2^*$ is the global maximizer) as
$\beta_2$ goes from $m(b(\beta_1))$ to $m(a(\beta_1))$. Thus there
must be a unique $\beta_2$: $m(b(\beta_1))<\beta_2<m(a(\beta_1))$
such that $u_1^*$ and $u_2^*$ are both global maximizers, and we
denote this $\beta_2$ by $r(\beta_1)$. The parameter values of $(\beta_1, r(\beta_1))$
are exactly the ones for which positive and negative areas bounded
by $L'(u)$ equal each other. An increase in $\beta_1$ induces an
upward shift of $L'(u)$, and may be balanced by a decrease in
$\beta_2$. Similarly, a decrease in $\beta_1$ induces a downward
shift of $L'(u)$, and may be balanced by an increase in $\beta_2$.
This justifies that $r(\beta_1)$ is monotonically decreasing in
$\beta_1$. See Figure \ref{Vshape}. Here we let $X$ be a random variable distributed according to Beta$(2, 2)$, then $\mathbb{E}(X)=1/2$ and $\Var(X)=1/20$. By Lemma \ref{KK}, $\theta_0=0$ and $u_0=\mathbb{E}(X)=1/2$, which by (\ref{fm}) gives $(\beta_1^c, \beta_2^c)=(-5, 5)$. Also see Figure 1 in \cite{RY} and Figure 1 in \cite{Yin2016} for related phase transition plots when the edge-weight distribution $\mu$ is respectively Bernoulli$(.5)$ and Uniform$(0, 1)$.

The rest of the proof follows as in the proof of the corresponding
result (Theorem 2.1) in Radin and Yin \cite{RY}, where some
probability estimates were used. A (jump) discontinuity in the
first derivatives of $\psi_\infty^\beta$ across the curve
$\beta_2=r(\beta_1)$ indicates a discontinuity in the expected
local densities, while the divergence of the second
derivatives of $\psi_\infty^\beta$ at the critical point
$(\beta_1^c, \beta_2^c)$ implies that the covariances of the local
densities go to zero more slowly than $1/n^2$. We omit
the proof details.

\begin{remark}
The maximization problem (\ref{lmax}) is solved at a unique value $u^*$ off the phase transition curve $\beta_2=r(\beta_1)$, and at two values $u_1^*$ and $u_2^*$ along the curve. As $\beta_1\rightarrow -\infty$ (resp. $\beta_2\rightarrow \infty$), $u_1^*\rightarrow 0$ and $u_2^*\rightarrow 1$. The jump from $u_1^*$ to $u_2^*$ is quite noticeable even for small parameter values of $\beta$. For example, taking $p=2$, $\beta_1=-8$, and $\beta_2=8$ in Beta$(2, 2)$, numerical computations yield that $u_1^* \approx 0.165$ and $u_2^* \approx 0.835$.
\end{remark}



\section{Universal asymptotics}
\label{asymptotics}
In this section we examine near degeneracy and universality in general edge-weighted exponential random graphs. All our findings in this section are derived based on the assumption that the non-degenerate probability measure $\mu$ for the edge weights is symmetric about the line $u=1/2$. We remark that near degeneracy and universality are expected even when the edge weights are not symmetrically distributed, except that the universal straight line gets shifted vertically from $\beta_2=-\beta_1$.

\begin{proposition}
\label{cor}
Consider a non-degenerate probability measure $\mu$ supported on $[0, 1]$ and symmetric about the line $u=1/2$. Take $H_1$ a single edge and $H_2$ a finite simple graph with
$p\geq 2$ edges. The phase transition curve
$\beta_2=r(\beta_1)$ lies above the straight line
$\beta_2=-\beta_1$ when $p\geq 3$, and is exactly the portion of the
straight line $\beta_2=-\beta_1$ ($\beta_1\leq -1/(4\Var(X))$ when $p=2$. Here $X$ is a random variable distributed according to $\mu$.
\end{proposition}

\begin{proof}
From the proof of Theorem \ref{phase}, there are two global
maximizers $u_1^*$ and $u_2^*$ for $L(u)$ along the phase
transition curve $\beta_2=r(\beta_1)$, $0<u_1^*<u_0<u_2^*<1$, where $u_0$ is the unique global minimizer for $m(u)$ (\ref{m}). By Lemma \ref{KK}, $u_0=1/2$ when $p=2$ and $u_0>1/2$ when $p>2$. Furthermore, the $y$-coordinate
$\beta_2^c$ of the critical point $(\beta_1^c, \beta_2^c)=(-f(u_0), m(u_0))$ is
always positive. On the straight line $\beta_1+\beta_2=0$, we
rewrite $L(u)=\beta_1(u-u^p)-I(u)/2$. By Proposition \ref{prop}, $I(u)$ is symmetric about the line $u=1/2$. First suppose
$p=2$. Since $I(u)$ and $u-u^2$ are both symmetric, two global maximizers $u_1^*$ and $u_2^*$ exist
for $L(u)$ and $(-f(u_0), m(u_0))=\left(-1/(4\Var(X)), 1/(4\Var(X))\right)$ by (\ref{fm}). Next consider the generic case $p\geq 3$. Analytical
calculations give that $u-u^p<(1-u)-(1-u)^p$ for
$0<u<1/2$. Since $I(u)$ is symmetric, this says that for
$\beta_1<0$ (resp. $\beta_2>0$), the global maximizer $u^*$ of
$L(u)$ satisfies $u^*\leq 1/2$ and so must be $u_1^*$. The conclusion readily follows.
\end{proof}

\begin{proposition}
\label{cor2}
Consider a non-degenerate probability measure $\mu$ supported on $[0, 1]$ and symmetric about the line $u=1/2$. Assume the associated Cram\'{e}r rate function (\ref{cramer}) is bounded on $[0, 1]$ (i.e. $I(0)=I(1)$ is finite). Take $H_1$ a single edge and $H_2$ a finite simple graph with $p\geq 2$ edges. The phase transition curve
$\beta_2=r(\beta_1)$ displays a universal asymptotic behavior as $\beta_1 \to -\infty$, specifically,
\begin{equation}
\label{erd}
\lim_{\beta_1 \to -\infty} \left|r(\beta_1) + \beta_1\right| = 0.
\end{equation}
\end{proposition}

\begin{proof}
Let $\beta_2 = -\beta_1 + \delta$ with $\delta>0$ fixed. Define $F(u; \beta_1) = \beta_1 (u-u^p)$ and $G(u; \delta) = \delta u^p - I(u)/2$ so that $L(u; \beta_1, \beta_2) = F(u; \beta_1) + G(u; \delta)$ by (\ref{smax}). We will show, for sufficiently negative $\beta_1$, that the global maximizer $u^*$ of $L(u)$ equals $u_2^*$. Together with Proposition \ref{cor}, this implies that for these $\beta_1$, $-\beta_1\leq r(\beta_1)\leq -\beta_1+\delta$, which will prove the desired limit.

Under our assumption, $-I(u)$ is a continuous symmetric function that increases on $(0, 1/2)$ and decreases on $(1/2, 1)$, with a maximum attained at $u=1/2$ and $-I(1/2)=0$. Denote by $C:=-I(0)/2=-I(1)/2$ so that $C$ is finite and negative and $G(0)=C$. Recall that $0<u_1^*<u_0<u_2^*<1$, where $u_1^*$ and $u_2^*$ are two local maximizers for $L(u)$ and $u_0 \geq 1/2$ is the unique global minimizer for $m(u)$ (\ref{m}) that does not depend on $\beta_1$ and $\beta_2$. Rigorously, it may be that only one local maximizer $u_1^*$ or $u_2^*$ exist for $L(u)$, but this does not affect our argument below. From the continuity and boundedness of $G$ on $[0, 1]$, there exists $\eta \in (0, 1-u_0)$ such that if $0\leq u<\eta$ then $G(u)-C<\delta/2$. Since $u-u^p=u(1-u^{p-1})>0$ on $(0, 1)$ and vanishes at the endpoints $0$ and $1$, there exists $\beta<0$ such that for all $\beta_1<\beta$ and $u \in [\eta, 1-\eta]$, $F(u)<C-\delta$ and therefore $L(u)<C-\delta+G(u)<C=L(0)$, so $u^* \in [0, \eta) \cup (1-\eta, 1]$. Similarly, using that $F(u) \leq 0$ for all $\beta_1<0$ and all $u\in [0, \eta)$, we have $L(u)\leq G(u)<C+\delta/2<C+\delta=L(1)$ so $u^* \in (1-\eta, 1]$. Since $u_1^*<u_0<1-\eta$, this says that $u^*=u_2^*$.
\end{proof}

Propositions \ref{cor} and \ref{cor2} have advanced our understanding of
phase transitions in edge-weighted exponential random graphs, yet some fundamental questions remain unanswered. As explained in Section \ref{maximization}, a typical graph sampled from the exponential model looks like an Erd\H{o}s-R\'{e}nyi graph $G(n, u)$ in the large $n$ limit, where the asymptotic edge presence probability $u(\beta_1, \beta_2) \rightarrow 0$ or $1$ is prescribed according to the maximization problem (\ref{lmax}). However, the speed of $u$ towards these two degenerate states is not at
all clear. When a typical graph is sparse ($u\rightarrow 0$),
how sparse is it? When a typical graph is nearly complete
$(u\rightarrow 1)$, how complete is it? Can we give an explicit
characterization of the near degenerate graph structure as a
function of the parameters? The following Theorems \ref{asymp1} and \ref{asymp2} are dedicated
towards these goals. Theorem \ref{asymp1} shows that $\theta$, the dual of the Erd\H{o}s-R\'enyi parameter $u$, displays universal asymptotic behavior in the sparse region of the parameter space ($\beta_1<-\beta_2$ and $\beta_2\geq 0$) whereas $u$ itself depends on the specific edge-weight distribution $\mu$. Theorem \ref{asymp2} provides a corresponding result in the nearly complete region of the parameter space ($\beta_1>-\beta_2$ and $\beta_2\geq 0$), showing that the dual $\theta$ again displays universal asymptotic behavior whereas the Erd\H{o}s-R\'enyi parameter $u$ still depends on the edge-weight distribution $\mu$.

\begin{theorem}
\label{asymp1}
Consider a non-degenerate probability measure $\mu$ supported on $[0, 1]$ and symmetric about the line $u=1/2$. Take $H_1$ a single edge and $H_2$ a finite simple graph with $p\geq 2$ edges. Let $\beta_1<-\beta_2$ and $\beta_2\geq 0$. For large $n$ and $(\beta_1, \beta_2)$ sufficiently far away from the origin, a typical graph drawn from the model looks like an Erd\H{o}s-R\'{e}nyi graph $G(n, u)$, where the edge presence probability $u$ depends on the distribution $\mu$, but its dual $\theta$ universally satisfies $\theta \asymp 2\beta_1$.
\end{theorem}

\begin{proof}
Let $\beta_1=a\beta_2$ with $a<-1$. Resorting to Legendre duality, (\ref{lmax}) gives a condition on $\theta$, the dual of $u$:
\begin{equation}
\label{condition}
\beta_1 + p\beta_2 (K'(\theta))^{p-1} = \frac{1}{2}\theta.
\end{equation}
By Proposition \ref{cor}, $u\rightarrow 0$ for $(\beta_1, \beta_2)$ sufficiently far away from the origin, which corresponds to $\theta \rightarrow -\infty$ in the dual space. From Table \ref{table}, $K'(\theta) \rightarrow 0$ as $\theta \rightarrow -\infty$, we have
\begin{equation}
\frac{\theta}{2\beta_2} = a + p(K'(\theta))^{p-1} \rightarrow a.
\end{equation}
The universal asymptotics of $\theta \asymp 2\beta_1$ is verified.

We claim that $u$ on the other hand depends on the specific distribution $\mu$. We will derive the asymptotics of $u$ in two special cases, Bernoulli$(.5)$ and Uniform$(0, 1)$. In both cases, $u=K'(\theta)$ by Legendre duality. For Bernoulli$(.5)$,
\begin{equation}
K'(\theta)=\frac{e^\theta}{1+e^\theta} \asymp e^\theta \asymp e^{2\beta_1}.
\end{equation}
While for Uniform$(0, 1)$,
\begin{equation}
K'(\theta)=\frac{e^\theta}{e^\theta-1}-\frac{1}{\theta} \asymp -\frac{1}{\theta} \asymp -\frac{1}{2\beta_1}.
\end{equation}
\end{proof}

\begin{theorem}
\label{asymp2}
Consider a non-degenerate probability measure $\mu$ supported on $[0, 1]$ and symmetric about the line $u=1/2$. Assume the associated Cram\'{e}r rate function (\ref{cramer}) is bounded on $[0, 1]$ (i.e. $I(0)=I(1)$ is finite). Take $H_1$ a single edge and $H_2$ a finite simple graph with $p\geq 2$ edges. Let $\beta_1>-\beta_2$ and $\beta_2\geq 0$. For large $n$ and $(\beta_1, \beta_2)$ sufficiently far away from the origin, a typical graph drawn from the model looks like an Erd\H{o}s-R\'{e}nyi graph $G(n, u)$, where the edge presence probability $u$ depends on the distribution $\mu$, but its dual $\theta$ universally satisfies $\theta \asymp 2(\beta_1+p\beta_2)$.
\end{theorem}

\begin{proof}
Let $\beta_1=a\beta_2$ with $a>-1$. Resorting to Legendre duality, (\ref{lmax}) gives condition (\ref{condition}) on $\theta$, the dual of $u$. By Proposition \ref{cor2}, $u\rightarrow 1$ for $(\beta_1, \beta_2)$ sufficiently far away from the origin, which corresponds to $\theta \rightarrow \infty$ in the dual space. From Table \ref{table}, $K'(\theta) \rightarrow 1$ as $\theta \rightarrow \infty$, we have
\begin{equation}
\frac{\theta}{2\beta_2} = a + p(K'(\theta))^{p-1} \rightarrow a+p.
\end{equation}
The universal asymptotics of $\theta \asymp 2(\beta_1+p\beta_2)$ is verified.

We claim that $u$ on the other hand depends on the specific distribution $\mu$. We will derive the asymptotics of $u$ in two special cases, Bernoulli$(.5)$ and Uniform$(0, 1)$. In both cases, $u=K'(\theta)$ by Legendre duality. For Bernoulli$(.5)$,
\begin{equation}
K'(\theta)=\frac{e^\theta}{1+e^\theta} \asymp 1-e^{-\theta} \asymp 1-e^{-2(\beta_1+p\beta_2)}.
\end{equation}
While for Uniform$(0, 1)$,
\begin{equation}
K'(\theta)=\frac{e^\theta}{e^\theta-1}-\frac{1}{\theta} \asymp 1-\frac{1}{\theta} \asymp 1-\frac{1}{2(\beta_1+p\beta_2)}.
\end{equation}
\end{proof}

See Tables \ref{table2} and \ref{table3}. Even for $\beta$ with small magnitude, the asymptotic tendency of the optimal $\theta$ (hence the optimal $u$) is quite evident. Here we take $p=2$. The asymptotic
characterizations of $u$ obtained in Theorems \ref{asymp1} and \ref{asymp2} make
possible a deeper analysis of the asymptotics of the limiting
normalization constant $\psi_\infty^\beta$ of the exponential model in the following Theorems \ref{lim1} and \ref{lim2}. Interestingly, universality is observed only in the nearly complete region ($\beta_1>-\beta_2$ and $\beta_2\geq 0$) of the parameter space as proven in Theorem \ref{lim2}, but not the sparse region ($\beta_1<-\beta_2$ and $\beta_2\geq 0$) as shown in Theorem \ref{lim1}.

Before stating the theorems and their proofs, we offer a possible explanation for this discrepancy. By Theorem \ref{CD},
\begin{equation}
\label{again}
\psi_\infty^\beta=\beta_1 u+\beta_2 u^p-\frac{1}{2}I(u),
\end{equation}
where $u$ is chosen so that the above equation is maximized. In statistical physics, $\beta_1 u+\beta_2 u^p$ is commonly referred to as the energy contribution and $-I(u)/2$ as the entropy contribution, with the latter being largely dependent on the specific edge-weight distribution $\mu$. In the sparse region of the parameter space, the entropy contribution is at least as important as the energy contribution and, for many common distributions such as Bernoulli$(.5)$ and Uniform$(0, 1)$ actually dominates the energy contribution. Conversely, in the nearly complete region of the parameter space, the energy contribution dominates the entropy contribution. This leads to universality of $\psi_\infty^\beta$ in the nearly complete region but not the sparse region.

\begin{table}
\begin{center}
\begin{tabular}{cccccc}
$\beta_1$ & $\beta_2$ & $\theta_{\text{opt}}$ & $u_{\text{opt}}$ & $\exp\left(2\beta_1\right)$ & $1-\exp\left(-2(\beta_1+p\beta_2)\right)$ \\
\hline \hline \\
$-2$ & $-4$ & $-4.23$ & $0.014$ & $0.018$  & \\
$1$ & $1$ & $5.99$ & $0.998$ & & $0.998$ \\ \\
\end{tabular}
\end{center}
\caption{Asymptotic comparison for Bernoulli$(.5)$ near degeneracy.} \label{table2}
\end{table}

\begin{table}
\begin{center}
\begin{tabular}{cccccc}
$\beta_1$ & $\beta_2$ & $\theta_{\text{opt}}$ & $u_{\text{opt}}$ & $-1/(2\beta_1)$ & $1-1/\left(2(\beta_1+p\beta_2)\right)$ \\
\hline \hline \\
$-4$ & $-6$ & $-10.32$ & $0.097$ & $0.125$  & \\
$3$ & $2$ & $13.40$ & $0.925$ & & $0.929$ \\ \\
\end{tabular}
\end{center}
\caption{Asymptotic comparison for Uniform$(0, 1)$ near degeneracy.} \label{table3}
\end{table}

\begin{theorem}
\label{lim1}
Consider a non-degenerate probability measure $\mu$ supported on $[0, 1]$ and symmetric about the line $u=1/2$. Take $H_1$ a single edge and $H_2$ a finite simple graph with $p\geq 2$ edges. Let $\beta_1<-\beta_2$ and $\beta_2\geq 0$. For $(\beta_1, \beta_2)$ sufficiently far away from the origin, the limiting normalization constant $\psi_\infty^\beta$ depends on the distribution $\mu$.
\end{theorem}

\begin{proof}
Let $\beta_1=a\beta_2$ with $a<-1$. Theorem \ref{CD} gives (\ref{again}), where $u$ is chosen so that the equation is maximized and $u \rightarrow 0$ for $(\beta_1, \beta_2)$ sufficiently far away from the origin. Resorting to Legendre duality, this gives
\begin{equation}
\psi_\infty^\beta=\beta_1 K'(\theta)+\beta_2 (K'(\theta))^p-\frac{1}{2} \left(\theta K'(\theta)-K(\theta)\right),
\end{equation}
where $\theta$ is the dual of $u$ and approaches $-\infty$ when $(\beta_1, \beta_2)$ diverge. By (\ref{condition}),
\begin{equation}
\label{sim}
\psi_\infty^\beta=(1-p)\beta_2 (K'(\theta))^p+\frac{1}{2}K(\theta).
\end{equation}
Since $\beta_2 \asymp \theta / (2a)$ as $\theta \rightarrow -\infty$ from Theorem \ref{asymp1}, asymptotically we have
\begin{eqnarray}
\psi_\infty^\beta &\asymp& \frac{1-p}{2a} \theta (K'(\theta))^p+\frac{1}{2}K(\theta)\\
&\asymp& (1-p) \beta_2 \left(K'(2\beta_1)\right)^p+\frac{1}{2}K(2\beta_1). \nonumber
\end{eqnarray}
\end{proof}

\begin{remark}
Many common distributions including Bernoulli$(.5)$ and Uniform$(0, 1)$ satisfy $\theta K'(\theta) / K(\theta) \rightarrow 0$ as $\theta \rightarrow -\infty$, in which case the asymptotics in Theorem \ref{lim1} may be further reduced to $\psi_\infty^\beta \asymp K(\theta)/2 \asymp K(2\beta_1)/2$.
\end{remark}

\begin{theorem}
\label{lim2}
Consider a non-degenerate probability measure $\mu$ supported on $[0, 1]$ and symmetric about the line $u=1/2$. Assume the associated Cram\'{e}r rate function (\ref{cramer}) is bounded on $[0, 1]$ (i.e. $I(0)=I(1)$ is finite). Take $H_1$ a single edge and $H_2$ a finite simple graph with $p\geq 2$ edges. Let $\beta_1>-\beta_2$ and $\beta_2\geq 0$. For $(\beta_1, \beta_2)$ sufficiently far away from the origin, the limiting normalization constant $\psi_\infty^\beta$ universally satisfies $\psi_\infty^\beta \asymp \beta_1+\beta_2$.
\end{theorem}

\begin{proof}
Let $\beta_1=a\beta_2$ with $a>-1$. Similarly as in the proof of Theorem \ref{lim1}, Theorem \ref{CD} gives (\ref{again}), where $u$ is chosen so that the equation is maximized and $u \rightarrow 1$ for $(\beta_1, \beta_2)$ sufficiently far away from the origin. Since the first two terms diverge to $\beta_1+\beta_2$ while the last term is bounded by our assumption, the claim easily follows.
\end{proof}

\begin{remark}
The boundedness assumption on $I$ in Theorem \ref{lim2} is only used as a sufficient condition to ensure that $u \rightarrow 1$ for $\beta_1>-\beta_2$ in the upper half-plane and far away from the origin and is not necessary for the derivation of the universal asymptotics for $\psi_\infty^\beta$. Indeed, since $\theta \asymp 2(\beta_1+p\beta_2)$ by Theorem \ref{asymp2}, using $K(\theta) / \theta \asymp K'(\theta) \asymp 1$ in (\ref{sim}), we have
\begin{equation}
\psi_\infty^\beta \asymp (1-p)\beta_2 +\left(\beta_1+p\beta_2\right) \asymp \beta_1+\beta_2.
\end{equation}
This universal asymptotic phenomenon is observed for example in Uniform$(0, 1)$, whose associated Cram\'{e}r rate function $I$ is not bounded.
\end{remark}

In the nearly complete region of the parameter space ($\beta_1>-\beta_2$ and $\beta_2\geq 0$) examined in Theorems \ref{asymp2} and \ref{lim2}, the ``asymptotically equivalent'' Erd\H{o}s-R\'{e}nyi parameter $u$ depends on the edge-weight distribution $\mu$ yet the limiting normalization constant $\psi_\infty^\beta$ for the exponential random graph displays universal asymptotic behavior. Since the Erd\H{o}s-R\'enyi model is not an exact statistical physics analog for the exponential model, this seemingly controversial discrepancy does not come as a surprise. We work out the details for standard $2$-parameter families with Bernoulli$(.5)$ edge-weight distribution below, and the calculation may be extended to $k$-parameter families with general edge-weight distributions.

Suppose the exponential random graph $G_n$ is indistinguishable in the large $n$ limit from an Erd\H{o}s-R\'{e}nyi random graph $G(n, u)$ in the graphon sense. In other words, for large $n$, the $2$-parameter exponential random graph $G_n$ is ``equivalent'' to a simplified $1$-parameter Erd\H{o}s-R\'{e}nyi random graph with probability distribution
\begin{equation}
\PR_n^{\beta'}(G_n)=\exp(n^2 (\beta' t(H_1,
G_n)-\psi_n^{\beta'})),
\end{equation}
where $u$ and $\beta'$ are related by $u=e^{2\beta'}/(1+e^{2\beta'})$. The limiting normalization constant
$\psi_\infty^{\beta'}$ for the Erd\H{o}s-R\'{e}nyi model is given by
\begin{equation}
\psi_\infty^{\beta'}=\log(1+e^{2\beta'})/2=-\log(1-u)/2,
\end{equation}
and the limiting normalization constant
$\psi_\infty^\beta$ for the exponential random graph model is given by
\begin{equation}
\psi_\infty^{\beta}=\beta_1 u+\beta_2 u^p-\frac{1}{2}u\log u-\frac{1}{2}(1-u)\log(1-u).
\end{equation}
Utilizing the fact that $u$ satisfies
\begin{equation}
\beta_1+p\beta_2(u)^{p-1}-\frac{1}{2}\log \frac{u}{1-u}=0,
\end{equation}
we have
\begin{equation}
\label{last}
\psi_\infty^\beta=(1-p)\beta_2u^p-\log(1-u)/2.
\end{equation}
This shows that $\psi_\infty^\beta$ (for the exponential random graph model) and $\psi_\infty^{\beta'}$ (for the corresponding Erd\H{o}s-R\'{e}nyi model) do not coincide unless $\beta_2=0$. The difference is particularly noticeable in the nearly complete region, where $\beta' \asymp \beta_1+p\beta_2$ when $\mu$ is Bernoulli$(.5)$, and so $\psi_\infty^{\beta'} \asymp \beta_1+p\beta_2$ but $\psi_\infty^\beta \asymp \beta_1+\beta_2$.

\section*{Acknowledgements}
The authors are very grateful to the anonymous referees for the invaluable suggestions that greatly improved the quality of this paper.


\begin{thebibliography}{30}

\bibitem{Aldous1} Aldous, D.: Representations for partially exchangeable arrays of random
variables. J. Multivariate Anal. 11, 581-598 (1981)

\bibitem{AL} Aldous, D., Lyons, R.: Processes on unimodular random networks. Electron. J. Probab. 12, 1454-1508 (2007)

\bibitem{AS} Aldous, D., Steele, J.M.: The objective method: Probabilistic combinatorial optimization and local weak convergence. In: Kesten, H. (ed.) Probability on Discrete Structures, pp. 1-72. Springer, Berlin (2004)

\bibitem{AZ2}
Aristoff, D., Zhu, L.: Asymptotic structure and singularities in
constrained directed graphs. Stochastic Process. Appl. 125,
4154-4177 (2015)

\bibitem{BS} Benjamini, I., Schramm, O.: Recurrence of distributional limits of finite planar graphs. Electron. J. Probab. 6, 1-13 (2001)

\bibitem{Besag} Besag, J.: Statistical analysis of non-lattice data. J. R. Stat. Soc. Ser. D. Stat. 24, 179-195 (1975)

\bibitem{BBS} Bhamidi, S., Bresler G., Sly A.: Mixing time of exponential random graphs. Ann. Appl. Probab. 21, 2146-2170 (2011)

\bibitem{BCCZ1} Borgs, C., Chayes, J., Cohn, H., Zhao, Y.:
An $L^p$ theory of sparse graph convergence I. Limits, sparse
random graph models, and power law distributions. arXiv: 1401.2906
(2014)

\bibitem{BCCZ2} Borgs, C., Chayes, J., Cohn, H., Zhao, Y.:
An $L^p$ theory of sparse graph convergence II. LD convergence,
quotients, and right convergence. arXiv: 1408.0744 (2014)

\bibitem{BCLSV1} Borgs, C., Chayes, J., Lov\'{a}sz, L., S\'{o}s, V.T., Vesztergombi, K.:
Counting graph homomorphisms. In: Klazar, M., Kratochvil, J.,
Loebl, M., Thomas, R., Valtr,  P. (eds.) Topics in Discrete
Mathematics, Volume 26, pp. 315-371. Springer, Berlin (2006)

\bibitem{BCLSV2} Borgs, C., Chayes, J.T., Lov\'{a}sz, L., S\'{o}s, V.T., Vesztergombi, K.:
Convergent sequences of dense graphs I. Subgraph frequencies,
metric properties and testing. Adv. Math. 219, 1801-1851 (2008)

\bibitem{BCLSV3} Borgs, C., Chayes, J.T., Lov\'{a}sz, L., S\'{o}s, V.T., Vesztergombi, K.:
Convergent sequences of dense graphs II. Multiway cuts and
statistical physics. Ann. of Math. 176, 151-219 (2012)

\bibitem{Brown} Brown, L.D.: Fundamentals of Statistical Exponential Families with Applications in Statistical Decision Theory. Lecture Notes-Monograph Series, Volume 9. Institute of Mathematical Statistics, Hayward (1986)

\bibitem{CD2} Chatterjee, S., Dembo, A.: Nonlinear large deviations. Adv. Math. 299, 396-450 (2016)

\bibitem{CD1} Chatterjee, S., Diaconis, P.: Estimating and understanding exponential random graph models. Ann. Statist. 41, 2428-2461 (2013)

\bibitem{CV} Chatterjee, S., Varadhan, S.R.S.: The large deviation principle for
the Erd\H{o}s-R\'{e}nyi random graph. European J. Combin. 32,
1000-1017 (2011)

\bibitem{CD} Cranmer, S.J., Desmarais, B.A.: Inferential network analysis with exponential random graph models. Pol. Anal. 19, 66-86 (2011)

\bibitem{Fienberg1} Fienberg, S.E.: Introduction to papers on the
modeling and analysis of network data. Ann. Appl. Statist. 4, 1-4
(2010)

\bibitem{Fienberg2} Fienberg, S.E.: Introduction to papers on the
modeling and analysis of network data II. Ann. Appl. Statist. 4,
533-534 (2010)

\bibitem{FS} Frank, O., Strauss, D.: Markov graphs. J. Amer. Statist. Assoc. 81, 832-842 (1986)

\bibitem{Handcock} Handcock, M.S.: Assessing degeneracy in statistical models of social networks. Working Paper 39, Center for Statistics and the Social Sciences, Univ. Washington, Seattle, WA. (2003)

\bibitem{Hoover} Hoover, D.: Row-column exchangeability and a generalized model for
probability. In: Koch, G., Spizzichino, F. (eds.) Exchangeability
in Probability and Statistics, pp. 281-291. North-Holland,
Amsterdam (1982)

\bibitem{Hunter} Hunter, D.R., Handcock, M.S., Butts, C.T., Goodreau, S.M., Morris, M.: ergm: A package to fit, simulate and diagnose exponential-family models for networks. J. Statist. Softw. 24, 1-29 (2008)

\bibitem{KRRS} Kenyon, R., Radin, C., Ren, K., Sadun, L.:
Multipodal structure and phase transitions in large constrained
graphs. arXiv: 1405.0599 (2014)

\bibitem{KY} Kenyon, R., Yin, M.: On the asymptotics of constrained exponential random graphs. J. Appl. Probab. 54, 165-180 (2017)

\bibitem{Lov} Lov\'{a}sz, L.: Large Networks and Graph Limits.
American Mathematical Society, Providence (2012)

\bibitem{LS} Lov\'{a}sz, L., Szegedy B.: Limits of
dense graph sequences. J. Combin. Theory Ser. B. 96, 933-957 (2006)

\bibitem{Lyons} Lyons, R.: Asymptotic enumeration of spanning trees. Combin. Probab. Comput. 14, 491-522 (2005)

\bibitem{LZ1} Lubetzky, E., Zhao, Y.: On replica symmetry of large deviations in random graphs. Random Structures Algorithms 47, 109-146 (2015)

\bibitem{LZ2} Lubetzky, E., Zhao, Y.: On the variational problem for upper tails in sparse random graphs. Random Structures Algorithms 50, 420-436 (2017)

\bibitem{M} Mukherjee, S.: Consistence estimation in the two star exponential random graph model. arXiv: 1310.4526 (2013)

\bibitem{Newman} Newman, M.: The structure and function of
complex networks. SIAM Rev. 45, 167-256 (2003)


\bibitem{RRS}
Radin, C., Ren, K., Sadun, L.: The asymptotics of large
constrained graphs. J. Phys. A: Math. Theor. 47, 175001 (2014)

\bibitem{RS1}
Radin, C., Sadun, L.: Phase transitions in a complex network. J.
Phys. A: Math. Theor. 46, 305002 (2013)

\bibitem{RS2}
Radin, C., Sadun, L.: Singularities in the entropy of
asymptotically large simple graphs. J. Stat. Phys. 158, 853-865
(2015)

\bibitem{RY} Radin, C., Yin, M.: Phase transitions in
exponential random graphs. Ann. Appl. Probab. 23, 2458-2471 (2013)

\bibitem{RFZ} Rinaldo, A., Fienberg, S.E., Zhou, Y.: On the geometry of discrete exponential families with application to exponential random graph models. Electron. J. Stat. 3, 446-484 (2009)

\bibitem{SPRH} Snijders, T.A.B., Pattison, P., Robins, G.L., Handcock, M.: New specifications for exponential random graph models. Sociol. Methodol. 36, 99-153 (2006)

\bibitem{WF} Wasserman, S., Faust, K.: Social Network Analysis: Methods and Applications. Cambridge University
Press, Cambridge (2010)

\bibitem{Yin2013} Yin, M.: Critical phenomena in exponential random graphs. J. Stat. Phys. 153, 1008-1021 (2013)

\bibitem{Yin} Yin, M.: A detailed investigation into near degenerate exponential random graphs. J. Stat.
Phys. 164, 241-253 (2016)

\bibitem{Yin2016} Yin, M.: Phase transitions in edge-weighted exponential random graphs. arXiv:
1607.04084 (2016)

\bibitem{ZRM} Zia, R.K.P., Redish, E.F., McKay, S.: Making sense of
the Legendre transform. Amer. J. Phys. 77, 614-622 (2009)

\end{thebibliography}
\end{document}